\documentclass[a4paper, reqno]{amsart}

\usepackage{enumerate,hyperref}

\usepackage{tikz}
\usetikzlibrary{shapes.geometric}
\usetikzlibrary{backgrounds,calc,positioning}
\usepackage[linguistics]{forest}
\usetikzlibrary{decorations.pathreplacing}
\usetikzlibrary{decorations.pathmorphing}

\theoremstyle{definition}
\newtheorem{definition}{Definition}[section]

\newtheorem{remark}[definition]{Remark}
\theoremstyle{plain}
\newtheorem{theorem}[definition]{Theorem}

\newtheorem{lemma}[definition]{Lemma}
\newtheorem{corollary}[definition]{Corollary}
\newtheorem{problem}[definition]{Problem}

\begin{document}

\title{Membership problems in finite groups}

\author[M.~Lohrey]{Markus Lohrey}
\email{lohrey@eti.uni-siegen.de}
\author[A.~Rosowski]{Andreas Rosowski}
\email{rosowski@eti.uni-siegen.de}
\author[G.~Zetzsche]{Georg Zetzsche}
\email{zetzsche@mpi-sws.org}

\address[Markus~Lohrey, Andreas Rosowski]{Universit{\"a}t Siegen, Germany}
\address[Georg~Zetzsche]{Max Planck Institute for Software Systems, Kaiserslautern, Germany}
\thanks{This work has been supported by the DFG research project LO 748/12-2}

\begin{abstract}
We show that the subset sum problem, the knapsack problem and the rational subset membership problem for permutation groups are {\bf NP}-complete. 
Concerning the knapsack problem we obtain {\bf NP}-completeness for every fixed $n \geq 3$, where $n$ is the number of permutations in the knapsack equation. In other words: membership in products of three cyclic permutation groups is  {\bf NP}-complete.  This sharpens a result of Luks \cite{luks}, which states {\bf NP}-completeness of the membership problem for products of three abelian permutation groups. We also consider the context-free membership problem in permutation groups and prove that it is {\bf PSPACE}-complete but {\bf NP}-complete for a restricted class of context-free grammars where acyclic derivation trees must have constant Horton-Strahler number. Our upper bounds hold for black box groups. The results for context-free membership problems 
in permutation groups yield new complexity bounds for various intersection non-emptiness problems for DFAs and a single context-free grammar.
\end{abstract}

\maketitle

\section{Introduction}

\subparagraph*{\bf Membership problems in groups.}
The general problem that we study in this paper is the following:
Fix a class $\mathcal{C}$ of formal languages. We assume that members of $\mathcal{C}$ have a finite description;
typical choices are the class of regular or context-free languages, or a singleton class $\mathcal{C} = \{L\}$.
We are given a language $L \in \mathcal{C}$ with $L \subseteq \Sigma^*$, a group 
$G$ together with a morphism $h : \Sigma^* \to G$ from the free monoid
$\Sigma^*$ to the group $G$, and a word $w \in \Sigma^*$. 
The question that we want to answer is whether $w \in h^{-1}(h(L))$, i.e., whether 
the group element $h(w)$ belongs to $h(L)$.
One can study this problem under several settings, and each of these settings has a different motivation. First of all, one can 
consider the case, where $G$ is a fixed finitely generated group that is finitely generated by $\Sigma$, and the input consists of $L$.
One could call this problem the $\mathcal{C}$-membership problem for the group $G$.
The best studied case is the {\em rational subset membership problem}, where 
$\mathcal{C}$ is the class of regular languages. 
It generalizes the subgroup membership problem for $G$, a classical decision problem in group theory. Other special cases of the 
rational subset membership problem that have been studied in the past are the submonoid membership problem, the knapsack problem
and the subset problem, see e.g.~\cite{lohrey_2015,MyNiUs14}. It is a simple observation that for the rational subset membership problem
one can assume that the word $w$ (that is tested for membership in $h^{-1}(h(L))$ can be assumed to be the empty word, see \cite[Theorem~3.1]{KaSiSt06}.

In this paper, we study another setting of the above generic problem, where $G$ is a finite group that is part of the input (and $L$ still 
comes from a  languages class $\mathcal{C}$). For the rest of the introduction we restrict to the case, where $G$ is a finite symmetric group $S_m$ (the set of all permutations on $\{1,\ldots,m\}$) that is represented in the input by the integer $m$ in unary representation, i.e., by the word $\$^m$.\footnote{We could 
also consider the case where $G$ is a subgroup of $S_m$ that is given by a list of generators (i.e., $G$ is a permutation group), but this
makes no difference for our problems.}
Our applications only make use of this case, but we 
remark that our upper complexity bounds can be proven in the more general black box setting  \cite{babai} (in particular, one could replace
symmetric groups by matrix groups over a finite field and still obtain the same complexity bounds). Note that $|S_m| = m!$, hence the order
of the group is exponential in the input length. 

\bigskip

\subparagraph*{\bf Membership problems for permutation groups.}

One of the best studied membership problems for permutation groups is the {\em subgroup membership problem}: 
the input is a unary encoded number $m$ and a list of permutations $a,a_1,\dots,a_n \in S_m$, and it is
asked whether $a$ belongs to the subgroup of $S_m$ generated by $a_1,\dots,a_n$.
The well-known Schreier-Sims algorithm solves this problem in polynomial time \cite{Sims70}, and the problem
is known to be in {\bf NC} \cite{BaLuSe87}.

Several generalizations of the subgroup membership problem have been studied. 
Luks defined the $k$-membership problem ($k \geq 1$) as follows: The input is a unary encoded number $m$,
a permutation $a\in S_m$ and a list of $k$ permutation groups $G_1, G_2, \ldots, G_k \leq S_m$ (every $G_i$ is given
by a list of generators). The question is
whether $a$ belongs to the product $G_1 \cdot G_2 \cdots G_k$.
It is a famous open problem whether $2$-membership can be solved in polynomial time. This problem is equivalent to several
other important algorithmic problems in permutation groups: computing the intersection of permutation groups,
computing set stabilizers or centralizers, checking equality of double cosets, see \cite{luks}  for details.
On the other hand, Luks has shown in \cite{luks} that $m$-membership is {\bf NP}-complete for every $k \geq 3$. 
In fact, {\bf NP}-hardness of $3$-membership holds for the special case where $G_1 = G_3$ and $G_1$ and $G_2$
are both abelian. 

Note that the  $k$-membership problem is a special case of the rational subset membership for symmetric groups.
Let us define this problem again for the setting of symmetric groups (here, $1$ denotes the identity permutation and we 
identify a word over the alphabet $S_m$ with the permutation to which it evaluates):

\begin{problem}[rational subset membership problem for symmetric groups]\label{rationalsubsetmemproblem}~\\
Input: a unary encoded number $m \in \mathbb{N}$ and a nondeterministic finite automaton (NFA) $\mathcal{A}$ over the alphabet $S_m$. \\
Question: Does $1 \in L(\mathcal{A})$ hold?
\end{problem}
An obvious generalization of the rational subset membership problem for symmetric groups is the 
{\em context-free subset membership problem for symmetric groups}; it is obtained by replacing the NFA $\mathcal{A}$ 
in Problem~\ref{rationalsubsetmemproblem} by a context-free grammar $\mathcal{G}$.

Two restrictions of the rational subset membership problem that have been intensively studied for infinite groups in recent
years are the {\em knapsack problem} and {\em subset sum problem}, see e.g. \cite{babai2,bell,BergstrasserGZ21,FigeliusGLZ20,FrenkelNU14,KoenigLohreyZetzsche2015a,LOHREY2019,LohreyZ18,MyNiUs14}.
For symmetric groups, these problems are defined as follows (note that the number $n+1$ of permutations is part of the input):

\begin{problem}[subset sum problem for symmetric groups]\label{subsetsumproblem}~\\
Input: a unary encoded number $m \in\mathbb{N}$ and permutations $a,a_1,\dots,a_n \in S_m$.\\
Question: Are there $i_1,\dots,i_n \in \{0,1\}$ such that $a=a_1^{i_1} \cdots a_n^{i_n}$?
\end{problem}
The subset sum problem is the membership problem for the cubes from \cite{babai}.

\begin{problem}[knapsack problem for symmetric groups]\label{knapsackproblem}~\\
Input: a unary encoded number $m \in\mathbb{N}$ and permutations $a,a_1,\dots,a_n \in S_m$.\\
Question: Are there $i_1,\dots,i_n \in \mathbb{N}$ such that $a=a_1^{i_1} \cdots a_n^{i_n}$?
\end{problem}

We will also consider the following restrictions of these problems.

\begin{problem}[abelian subset sum problem for symmetric groups]\label{abeliansubsetsum}~\\
Input: a unary encoded number $m \in\mathbb{N}$ and pairwise commuting permutations $a,a_1,\dots,a_n \in S_m$.\\
Question: Are there $i_1,\dots,i_n \in \{0,1\}$ such that $a=a_1^{i_1} \cdots a_n^{i_n}$?
\end{problem}
The following problem is the special case of Luks' $k$-membership problem for cyclic groups. Note that
$k$ is a fixed constant here.

\begin{problem}[$k$-knapsack problem for symmetric groups]\label{nknapsackproblem}~\\
Input: a unary encoded number $m \in\mathbb{N}$ and $k+1$ permutations $a,a_1,\dots,a_k \in S_m$. \\
Question: Are there $i_1,\dots,i_k \in \mathbb{N}$ such that $a=a_1^{i_1} \cdots a_k^{i_k}$?
\end{problem}

\bigskip

\subparagraph*{\bf Main results.}

Our main result for the rational subset membership problem in symmetric groups is:
\begin{theorem}\label{theoremnpcomplete}
Problems~\ref{rationalsubsetmemproblem}--\ref{abeliansubsetsum} and Problem~\ref{nknapsackproblem} for $k \geq 3$
are {\bf NP}-complete. 
\end{theorem}
In contrast, we will show that the $2$-knapsack problem can be solved in polynomial time (Theorem~\ref{thm-2-knapsack}).
The {\bf NP} upper bound for the rational subset membership problem will be shown for black-box groups.

\begin{remark}
The abelian variant of the knapsack problem, i.e.,  Problem~\ref{knapsackproblem} with the additional restriction that the permutations
$s_1,\dots,s_n$ pairwise commute is of course the abelian subgroup membership problem and hence belongs to {\bf NC}.
\end{remark}

\begin{remark}
Analogously to the $k$-knapsack problem one might consider the $k$-subset sum problem, where the number $n$ in Problem~\ref{subsetsumproblem}
is fixed to $k$ and not part of the input. This problem can be solved in time $2^k \cdot m^{\mathcal{O}(1)}$ (check all $2^k$ assignments for exponents $i_1, \ldots, i_k$)
and hence in polynomial time for every fixed $k$.
\end{remark}

Finally, for the context-free subset membership problem for symmetric groups we show:
\begin{theorem}\label{theoremPSPACEcomplete}
The context-free membership problem for symmetric groups is {\bf PSPACE}-complete.
\end{theorem}
If we restrict the class of context-free grammars in Theorem~\ref{theoremPSPACEcomplete} we can improve
the complexity to {\bf NP}: A derivation tree of a context-free grammar is called {\em acyclic} if no nonterminal
appears twice on a path from the root to a leaf. Hence, the height of an acyclic derivation tree is bounded by
the number of nonterminals of the grammar. The {\em Horton-Strahler number} $\text{hs}(t)$ of a binary tree $t$ (introduced by Horton and Strahler in the context of hydrology \cite{Ho45,Stra52}; see 
 \cite{EsparzaLS14} for a good survey emphasizing the importance of Horton-Strahler numbers in computer science)  is recursively
defined as follows: If $t$ consists of a single node then $\text{hs}(t)=0$. Otherwise, assume that $t_1$ and $t_2$ are the subtrees 
rooted in the two children of the node. If $\text{hs}(t_1) = \text{hs}(t_2)$ then $\text{hs}(t) = 1 + \text{hs}(t_1)$, and if $\text{hs}(t_1) \neq \text{hs}(t_2)$ then 
$\text{hs}(t) = \max\{\text{hs}(t_1), \text{hs}(t_2)\}$.
For $k \geq 1$ let $\text{CFG}(k)$ be the set of all context-free grammars in Chomsky normal form (hence, derivation trees are
binary trees if we ignore the leafs labelled with terminal symbols) such that every acyclic derivation tree has Horton-Strahler number at most $k$.

\begin{theorem}\label{theoremStrahler}
For every $k \ge 1$, the context-free membership problem for symmetric groups restricted to context-free grammars from $\text{CFG}(k)$ is  {\bf NP}-complete.
\end{theorem}
Note that this result generalizes the statement for the rational subset membership problem in Theorem~\ref{theoremnpcomplete} since 
every regular grammar (when brought into Chomsky normal form) belongs to $\text{CFG}(1)$. Also linear context-free grammars belong to $\text{CFG}(1)$.
We remark that Theorem~\ref{theoremStrahler} is a promise problem in the sense that {\bf coNP} is the best upper complexity bound 
for testing whether a given context-free
grammar belongs to $\text{CFG}(k)$ that we are aware of; see the appendix.

The upper bounds in Theorems~\ref{theoremnpcomplete}, \ref{theoremPSPACEcomplete}, and \ref{theoremStrahler} 
will be actually shown for black box groups.

\bigskip

\subparagraph*{\bf Application to intersection non-emptiness problems.}

We can apply Theorems~\ref{theoremPSPACEcomplete} and \ref{theoremStrahler} to intersection non-emptiness problems.
The {\em intersection non-emptiness problem for deterministic finite automata} (DFAs) is the following problem:

\begin{problem}[intersection non-emptiness problem for DFAs]~\\
Input: DFAs $\mathcal{A}_1, \mathcal{A}_2, \ldots, \mathcal{A}_n$ \\
Question: Is $\bigcap_{1 \le i \le n} L(\mathcal{A}_i)$ non-empty?
\end{problem}
Kozen \cite{Kozen77} has shown that this problem is {\bf PSPACE}-complete.
When restricted to group DFAs (see Section~\ref{sec-prel}) the 
intersection non-emptiness problem was shown to be {\bf NP}-complete by Blondin et al.~\cite{BlondinKM16}. 
Based on Cook's characterization of {\bf EXPTIME} by polynomially space bounded AuxPDAs \cite{Cook71},
Swernofsky and Wehar \cite{SwernofskyW15} showed that the intersection non-emptiness problem is {\bf EXPTIME}-complete\footnote{The intersection non-emptiness problem becomes undecidable if one allows more than one context-free grammar.}
for a list of general DFAs  and a single context-free grammar; see also \cite[p.~275]{HeussnerLMS10} and see \cite{EsparzaKS03} for a related
{\bf EXPTIME}-complete problem.
Using Theorems~\ref{theoremPSPACEcomplete} and \ref{theoremStrahler} we can easily show the following new results:
\begin{theorem}\label{theorem-intersection-CFL-DFA(k)}
The following problem is {\bf NP}-complete for every $k \geq 1$:

\medskip
\noindent
Input: A list of group DFAs $\mathcal{A}_1, \mathcal{A}_2, \ldots, \mathcal{A}_n$ and a context-free grammar $\mathcal{G} \in \text{CFG}(k)$.\\
Question: Is $L(\mathcal{G}) \cap \bigcap_{1 \le i \le n} L(\mathcal{A}_i)$ non-empty?
\end{theorem}

\begin{theorem}\label{theorem-intersection-CFL-DFA}
The following problem is {\bf PSPACE}-complete:

\medskip
\noindent
Input: A list of group DFAs $\mathcal{A}_1, \mathcal{A}_2, \ldots, \mathcal{A}_n$ and a context-free grammar $\mathcal{G}$.\\
Question: Is $L(\mathcal{G}) \cap \bigcap_{1 \le i \le n} L(\mathcal{A}_i)$ non-empty?
\end{theorem}
Table~\ref{table} gives an overview on the complexity of intersection non-emptiness problems. For the intersection non-emptiness problem
for arbitrary DFAs and one grammar from $\text{CFG}(k)$ one has to notice that in the {\bf EXPTIME}-hardness proof from \cite{SwernofskyW15} 
one can choose a fixed context-free grammar. Moreover, every fixed context-free grammar belongs to $\text{CFG}(k)$ for some $k \geq 1$.

\begin{table}[t]
\begin{center}
  \begin{tabular}{ r | c | c | c |  }
                    & no CFG                      & one CFG$(k)$      & one CFG \\ \hline
DFAs           & {\bf PSPACE}-c.~\cite{Kozen77} & {\bf EXPTIME}-c.~for $k$ large enough          & {\bf EXPTIME}-c.~\cite{SwernofskyW15} \\ \hline
group DFAs & {\bf NP}-c.~\cite{BlondinKM16}        & {\bf NP}-c.~for all $k \geq 1$ & {\bf PSPACE}-c. \\
    \hline
  \end{tabular}
\end{center}
\caption{\label{table} Complexity of various intersection non-emptiness problems}
\end{table}

\bigskip

\subparagraph*{\bf Related work.}
Computationally problems for permutation groups
have a long history (see e.g.~the text book \cite{seress03}), and have applications, e.g. for graph isomorphism testing \cite{Babai16,Luks82}.
A problem that is similar to subset sum is the {\em minimum generator sequence problem} (MGS) \cite{EvenG81}: The input consists of
unary encoded numbers $m, \ell$ and a list of permutations $a,a_1,\dots,a_n \in S_m$. The question is, whether
$a$ can be written as a product $b_1 b_2 \cdots b_k$ with $k \leq \ell$ and $b_1, \ldots, b_k \in \{a_1, \ldots, a_n\}$.
The problem MGS was shown to be {\bf NP}-complete in \cite{EvenG81}. For the case, where the number $\ell$ is given
in binary representation, the problem is {\bf PSPACE}-complete \cite{Jerrum85}. This yields in fact the {\bf PSPACE}-hardness in 
Theorem~\ref{theoremPSPACEcomplete}.

Intersection nonemptiness problems for finite automata have been studied intensively in recent years, see e.g. \cite{ArrighiF00JO021,OliveiraW20}.
The papers \cite{birget2000pspace,jack2021complexity} prove {\bf PSPACE}-hardness of the intersection nonemptiness problem
for inverse automata (DFAs, where the transition monoid is an inverse monoid).

Horton-Strahler numbers have been used in the study of context-free languages before, see \cite{EsparzaLS14} for further information and references.

\section{Preliminaries} \label{sec-prel}

\subparagraph*{\bf Groups.}
Let $G$ be a finite group and let $G^*$ be the free monoid of all finite words over the alphabet $G$.
There is a canonical morphism $\phi_G : G^* \to G$ that is the identity mapping on $G \subseteq G^*$.
Throughout this paper we suppress applications of $\phi_G$ and identify words over the alphabet $G$ with the 
corresponding group elements.
For a subset $S \subseteq G$ we denote with $\left\langle S \right\rangle$ the subgroup generated by $S$.
The following folklore lemma is a straightforward consequence of Lagrange's theorem
(if $A$ and $B$ are subgroups of $G$ with $A < B$, then $|B| \geq 2 \cdot |A|$).

\begin{lemma} \label{lemma-lagrange}
Let $G$ be a finite group and $S \subseteq G$ a generating set for $G$. Then, there exists a subset $S' \subseteq S$
such that $\left\langle S' \right\rangle = G$ and  $|S'| \leq \log_2 |G|$.
\end{lemma}
Assume that $G = \langle S \rangle$.
A {\em straight-line program} over the generating set $S$ is a sequence of definitions $\mathcal{S} = (x_i := r_i)_{1 \leq i \leq n}$ where the $x_i$
are variables and every right-hand side $r_i$ is either from $S$ or of the form $x_j x_k$ with $1 \le j,k < i$.
Every variable $x_i$ evaluates to a group element $g_i \in G$ in the obvious way: if $r_i \in S$ then $g_i = r_i$ and if 
$r_i = x_j x_k$ then $g_i = g_j g_k$. We say that $\mathcal{S}$ produces $g_n$. The size of $\mathcal{S}$ is $n$.
The following result is known as the reachability theorem from \cite[Theorem~3.1]{babai}.

\begin{theorem}[reachability theorem] \label{reachability-theorem}
Let $G$ be a finite group, $S \subseteq G$ a generating set for $G$, and $g \in G$.
Then there exists a straight-line program over $S$ of size at most $(1+\log_2 |G|)^2$ that produces the element $g$.
\end{theorem}
For a set $Q$ let $S_Q$ be the symmetric group on $Q$, i.e., the set of all permutations on $Q$ with composition of permutations as the group operation.
If $Q = \{1,\ldots,m\}$ we also write $S_m$ for $S_Q$.
Let $a \in S_Q$ be a permutation and let $q \in Q$. We also write $q^a$ for $a(q)$.
We multiply permutations from left to right, i.e., for $a,b \in S_Q$, $ab$ is the permutation with
$q^{ab} = (q^a)^b$ for all $q \in Q$. 
A permutation group is a subgroup of some $S_Q$.

Quite often, the permutation groups we consider are actually
direct products $\prod_{1 \leq i \leq d} S_{m_i}$ for small numbers $m_i$.
Clearly, we have $\prod_{1 \leq i \leq d} S_{m_i} \leq S_m$ for $m = \sum_{1 \leq i \leq d} m_i$ and an embedding of
$\prod_{1 \leq i \leq d} S_{m_i}$ into $S_m$ can be computed in polynomial time.

\bigskip

\subparagraph*{\bf Horton-Strahler number.}
Recall the definition of the Horton-Strahler number $\text{hs}(t)$ of a binary tree $t$ from the introduction. We need the following 
simple fact, where we define the height of a binary tree as the maximal number of edges on a path from the root to a leaf.

\begin{lemma} \label{lemma-strahler}
Let $t$ be a binary tree of height $d$ and let $s = \text{hs}(t)$. Then, $t$ has at most $d^s$ many leaves and therefore at most $2 \cdot d^s$
many nodes.
\end{lemma}

\begin{proof}
We prove the statement by induction on $s = \text{hs}(t)$. If $s=0$ then $t$ must consist of a single leaf and the statement holds (we define $0^0 = 1$).
Otherwise take a path $v_1, v_2, \ldots, v_k$ in $t$, where $v_1$ is the root, $v_k$ is a leaf, and for every $2 \le i \le k$, if $t_i$ is the subtree rooted in $v_i$
and $t'_i$ is the subtree rooted in the sibling node of $v_i$, then $\text{hs}(t_i) \geq \text{hs}(t'_i)$.  Let $t_1 = t$. Then we must have $\text{hs}(t'_{i+1}) < \text{hs}(t_i) \leq s$ 
for every $1 \leq i \leq k-1$. Moreover, every $t'_i$ has height at most $d-1$. Using induction, we can bound the number of leaves in $t$
by
$$
1 + \sum_{i=2}^k  (d-1)^{s-1} \leq 1 + d \cdot (d-1)^{s-1} \leq  d^s .
$$
This shows the lemma.
\end{proof}

\bigskip

\subparagraph*{\bf Formal languages.} We assume that the reader is familiar with basic definitions from automata theory. Our definitions
of deterministic finite automata (DFA), nondeterministic finite automata (NFA), and context-free grammars are the standard ones.

Consider a DFA $\mathcal{A} = (Q, \Sigma, q_{0}, \delta, F)$, where $q_0 \in Q$ is the initial state, $\delta : Q \times \Sigma \to Q$ is the transition
mapping and $F \subseteq Q$ is the set of final states. The {\em transformation monoid} of $\mathcal{A}$ is the submonoid of $Q^Q$ (the set of all
mappings on $Q$ and composition of functions as the monoid operation) generated by all mappings $q \mapsto \delta(q,a)$ for $a \in \Sigma$.
A {\em group DFA} is a DFA whose transformation monoid is a group.

Context-free grammars will be always in Chomsky normal form.
When we speak of a {\em derivation tree} of a context-free grammar, we always assume that the root of the tree is labelled with the start nonterminal
and every leaf is labelled with a terminal symbol. When we talk about the Horton-Strahler number of such a tree, we remove all terminal-labelled
leafs so that the resulting tree is a binary tree (due to the Chomsky normal form).  In a {\em partial derivation tree}, we also allow leafs labelled with
nonterminals (but we still assume that the root is labelled with the start nonterminal).

\section{Black box groups}

More details on black box groups can be found in \cite{babai,seress03}.
Roughly speaking, in the black box setting group elements are encoded by bit strings of a certain length $b$ and there exist oracles for multiplying two group elements, computing the inverse of a group element, checking whether a given group element is the identity, and checking whether a given bit string of length $b$ is a valid encoding of a group element.\footnote{The latter operation is not 
allowed in \cite{babai}.}
As usual, each execution of an oracle operation counts one time unit, but the parameter $b$ enters the input length
additively.

Formally, a {\em black box} is a tuple 
$$
B=(b,c,\mathsf{valid},\mathsf{inv},\mathsf{prod},\mathsf{id},G,f),
$$
such that $G$ is a finite group (the group in the box), $b, c \in \mathbb{N}$,
and the following properties hold:
\begin{itemize}
\item $f : \{ 0,1\}^b \to G \uplus \{\ast\}$ is a mapping with $$G \subseteq f(\{ 0,1\}^b)$$
($f^{-1}(g) \neq \emptyset$ is the set names of the group element $g$).
\item $\mathsf{valid} : \{0,1\}^b \to  \{ \mathsf{yes}, \mathsf{no}\}$ is a mapping such that
$$\forall x \in \{0,1\}^b : f(x) \in G \ \Longleftrightarrow  \ \mathsf{valid}(x)=\mathsf{yes}.$$
\item $\mathsf{inv} : \{0,1\}^b \to \{0,1\}^b$ is a mapping such that for all $x \in f^{-1}(G)$:  $$f(\mathsf{inv}(x)) = f(x)^{-1}.$$
\item $\mathsf{prod} : \{0,1\}^b \times \{0,1\}^b \to \{0,1\}^b$ is a mapping such that for all $x,y \in f^{-1}(G)$: $$f(\mathsf{prod}(x,y)) = f(x) f(y).$$
\item $\mathsf{id} : \{0,1\}^b \times \{0,1\}^{c} \to \{ \mathsf{yes}, \mathsf{no}\}$ is a mapping such that
for all $x \in f^{-1}(G)$: 
$$f(x) = 1 \ \Longleftrightarrow \ \exists y \in \{0,1\}^{c} : \mathsf{id}(x,y) = \mathsf{yes}$$
 (such a $y$ is called a witness for $f(x)=1$).
\end{itemize}
We call $b$ the code length of the black box. 

A black box Turing machine is a deterministic or nondeterministic oracle Turing machine $T$ that has four special oracle
query states $q_{\mathsf{valid}}$, $q_{\mathsf{inv}}$,  $q_{\mathsf{prod}}$, $q_{\mathsf{id}}$, together with a special oracle tape, on which
a binary string is written.  The input for $T$ consists of two unary encoded numbers $b$ and $c$ and some additional
problem specific input. In order to determine the behavior of $T$ on the four special states $q_{\mathsf{valid}}$, $q_{\mathsf{inv}}$,  $q_{\mathsf{prod}}$, $q_{\mathsf{id}}$,
$T$  must be coupled with 
a black box $B = (b,c,\mathsf{valid},\mathsf{inv},\mathsf{prod},\mathsf{id},G,f)$ (where $b$ and $c$ must match the first part of the input of $T$). 
Then $T$ behaves as follows:
\begin{itemize}
\item After entering $q_{\mathsf{valid}}$ the oracle tape is overwritten by $\mathsf{valid}(x)$ where $x \in  \{0,1\}^b$ is the bit string consisting of the first
$b$ bits on the oracle tape.
\item After entering $q_{\mathsf{inv}}$ the oracle tape is overwritten by $\mathsf{inv}(x)$ where $x \in  \{0,1\}^b$ is the bit string consisting of the first
$b$ bits on the oracle tape.
\item After entering $q_{\mathsf{prod}}$ the oracle tape is overwritten by $\mathsf{prod}(x,y)$ where $x,y \in \{0,1\}^{b}$ and
$xy$ is the bit string consisting of the first $2b$ bits on the oracle tape.
\item After entering $q_{\mathsf{id}}$ the oracle tape is overwritten by $\mathsf{id}(x,y)$ where $x \in \{0,1\}^{b}$, $y \in \{0,1\}^{c}$ and
$xy$ is the bit string consisting of the first $b+c$ bits on the oracle tape.
\end{itemize}
As usual with oracle Turing machines, each of these four operations happens instantaneously and counts time $\mathcal{O}(1)$ for the total running time. 
We denote the machine with the above behaviour by $T_B$. Note that the black box 
$$B = (b,c,\mathsf{valid},\mathsf{inv},\mathsf{prod},\mathsf{id},G,f)$$ is not part of the input of $T$, only the unary encoded numbers $b$ and $c$
are part of the input.

Assume that $\mathcal{P}$ is an algorithmic problem for finite groups. The input for $\mathcal{P}$ is a finite group $G$ and some additional data $X$ (e.g. a context-free 
grammar with terminal alphabet $G$ in the next section). We do not specify exactly, how $G$ is represented. The additional input $X$ may contain elements of $G$.
We will say that $\mathcal{P}$  belongs to {\bf NP} for black box groups
if there is a nondeterministic black box Turing machine $T$, whose input is of the form $b,c,X$ with unary encoded numbers $b$ and $c$, 
such that for every black box 
$B = (b,c,\mathsf{valid},\mathsf{inv},\mathsf{prod},\mathsf{id},G,f)$ the following holds: $T_B$ accepts the input $b,c,X$ 
(where $X$ denotes the additional input for $\mathcal{P}$ and group elements in $X$ are represented by bit strings from $f^{-1}(G)$) if and only if $(G,X)$ belongs to $\mathcal{P}$.
The running time of $T_B$ is polynomial in $b+c+|X|$. Note that since $G$ may have order $2^b$, the order of $G$ may be exponential in the input length.
We will use the analogous definition for other complexity classes, in particular for {\bf PSPACE}.

For the rest of the paper we prefer a slightly more casual handling of black box groups. We always identify bits strings from $x \in f^{-1}(G)$ with 
the corresponding group elements.  So, we will never
talk about bit strings $x \in f^{-1}(G)$, but instead directly deal with elements of $G$. The reader should notice that we cannot directly verify whether
a given element $g \in G$ is the identity. This is only possible in a nondeterministic way by guessing a witness $y \in \{0,1\}^{c}$.
The same applies to the verification of an identity $g = h$, which is equivalent to $g h^{-1} = 1$.
This allows to cover also quotient groups by the black box setting; see \cite{babai}.

We need the following well-known fact from \cite{babai}:
\begin{lemma} \label{lemma-gwp-black-box}
The subgroup membership problem for  black box groups (given group elements $g, g_1, \ldots, g_n$, does $g \in \langle g_1, \ldots, g_n \rangle$ hold?)
belongs to {\bf NP}.
\end{lemma}
This is a consequence of the reachability theorem: Let $b$ be the code length of the black box. Hence there are at most $2^b$ group elements.
By the reachability theorem (Theorem~\ref{reachability-theorem}) it suffices to  
guess a straight-line program over $\{ g_1, \ldots, g_n \}$ of size at most $(1+\log_2 2^b)^2 = (b+1)^2$, evaluate it using the oracle for $\mathsf{prod}$ (let $g'$
be the result of the evaluation) and check whether $g' g^{-1} = 1$. The later can be done nondeterministically using the oracle for $\mathsf{id}$.

\section{Context-free membership in black box groups} \label{sec-cfg}

The goal of this section is to prove the following two results. Recall the definition of the class $\text{CFG}(k)$ from the introduction.

\begin{theorem}\label{theoremPSPACEblack-box}
The context-free subset membership problem for black box groups is in {\bf PSPACE}.
\end{theorem}

\begin{theorem}\label{theoremStrahler-black-box}
For every $k \ge 1$, the context-free membership problem for black box groups restricted to context-free grammars from $\text{CFG}(k)$ is in  {\bf NP}.
\end{theorem}

Before we prove these results, let us derive some corollaries. Theorem~\ref{theoremStrahler} is a direct corollary of Theorem~\ref{theoremStrahler-black-box}.
Restricted to regular grammars (which are in  $\text{CFG}(1)$ after bringing them to Chomsky normal form) we get:

\begin{corollary} \label{thm-rat-black-box}
The rational subset membership problem for black box groups is in {\bf NP}.
In particular, the rational subset membership problem for symmetric groups is in {\bf NP}.
\end{corollary}
Also Theorem~\ref{theoremPSPACEcomplete} can be easily obtained now: 
The upper bound follows directly from Theorem~\ref{theoremPSPACEblack-box}.
The lower bound can be obtained from a result of Jerrum \cite{Jerrum85}. In the introduction we mentioned that 
Jerrum proved the  {\bf PSPACE}-completeness of the MGS problem for the case where the number $\ell$ is give in binary notation. Given
permutations $a_1, \ldots, a_n \in S_m$ and a binary encoded number $\ell$ one can easily construct a context-free grammar for 
$\{1,a_1,\ldots,a_n\}^{\ell} \subseteq S_m$. Hence, the MGS problem with $\ell$ given in binary notation reduces to the context-free membership problem for symmetric groups,
showing that the latter is {\bf PSPACE}-hard. 

In the rest of the section we prove Theorems~\ref{theoremPSPACEblack-box} and \ref{theoremStrahler-black-box}.
We fix a finite group $G$ that is only accessed via a black box. 

\subparagraph*{The spanning tree technique.}

We start with subgroups of $G$ that are defined by finite nondeterministic automata (later, we will apply the following construction to a different group 
that is also given via a black box).
Assume that $\mathcal{A} = (Q, G, \{q_0\}, \delta, \{q_0\})$ is a finite nondeterministic automaton with terminal alphabet $G$. Note 
that $q_0$ is the unique initial and the unique final state. This ensures that the language $L(\mathcal{A})$ defined by $\mathcal{A}$
(which, by our convention, is identified with a subset of the group $G$) is a subgroup of $G$: the set $L(\mathcal{A})$ is clearly a submonoid
and every submonoid of a finite group is a subgroup. We now show a classical technique for finding a generating set for $L(\mathcal{A})$.

In a first step we remove from $\mathcal{A}$ all states $p \in Q$ such that there is no path from $q_0$ to $p$ as well as all
states $p$ such that there is no path from $p$ to $q_0$.  Let $\mathcal{A}_1$ be the resulting NFA.
We have $L(\mathcal{A}) = L(\mathcal{A}_1)$. 

In the second step we add for every transition $(p,g,q)$ the inverse transition $(q,g^{-1},p)$ (unless it already exists).
Let $\mathcal{A}_2$ be the resulting NFA. We claim that $L(\mathcal{A}_1) = L(\mathcal{A}_2)$. 
Note that by the first step, there must be a path from $q$ to $p$ in $\mathcal{A}_1$.
Let $h \in G$ be the group element produced by this path.
Take a $k > 0$ such that $(gh)^k = 1$ in $G$. Hence, $g^{-1} = h (gh)^{k-1}$. Moreover, there is a path in $\mathcal{A}_1$
from $q$ to $p$ which produces
the group element $h (gh)^{-1} = g^{-1}$. This shows that $L(\mathcal{A}_1) = L(\mathcal{A}_2)$. 

In the third step we compute the generating set for $L(\mathcal{A}_2) = L(\mathcal{A})$ using the spanning tree technique (see \cite{KM02} for an application in the context of free groups).
Consider the automaton $\mathcal{A}_2$ as an undirected multi-graph $\mathcal{G}$. The nodes of $\mathcal{G}$ are the states of $\mathcal{A}_2$.
Moreover, every undirected pair $\{ (p,g,q), (q,g^{-1},p) \}$ of transitions in the  NFA $\mathcal{A}_2$ is 
an undirected edge in $\mathcal{G}$ connecting the nodes $p$ and $q$. Note that there can be several
edges between two nodes (as well as loops); hence $\mathcal{G}$ is indeed a multi-graph.
We then compute a spanning tree $\mathcal{T}$ of $\mathcal{G}$. For every state of $p$ of $\mathcal{A}_2$ we fix
a  directed simple path $\pi_p$ in $\mathcal{T}$ from $q_0$ to $p$. We can view this path $\pi_p$ as a path in $\mathcal{A}_2$.
Let $g_p$ be the group element produced by the path $\pi_p$.
For every undirected edge $e = \{ (p,g,q), (q,g^{-1},p) \}$ in $\mathcal{G} \setminus \mathcal{T}$ 
let $g_e := g_p g g_q^{-1}$ (we could also take $g_q g^{-1} g_p^{-1}$).
A standard argument shows that the set $\{ g_e \mid e \text{ is an edge in } \mathcal{G} \setminus \mathcal{T}\}$
indeed generates $L(\mathcal{A})$.

The above construction can be carried out in polynomial time for black box groups. This is straightforward. The only detail that we 
want to emphasize is that we have to allow multiple copies of undirected edges $\{ (p,g,q), (q,g^{-1},p) \}$ in the black box setting.
The reason is that we may have several names (bit strings) denoting the same group element and we can only 
verify nondeterministically whether two bit strings represent the same group element. But this is not a problem;
it just implies that we may output copies of the same generator.

\subparagraph*{The operations $\Delta$ and $\Gamma$.}
Let $\mathcal{G}=(N,G,P,S)$ be a context-free grammar in Chomsky normal form that is part of the input, whose terminal
alphabet is the finite group $G$.  When we speak of the input size in the following, we refer to $|\mathcal{G}|+b+c$, where $b$ and $c$ are the two
unary encoded numbers from the black box for $G$ and the size $|\mathcal{G}|$ is defined 
as the number of productions of the grammar.

With $L(A)$ we denote the set of all words $w \in G^*$ that are derived from the nonterminal $A \in N$ and, as usual,
we identify $L(A)$ with $\phi_G(L(A)) \subseteq G$. 
Let $\hat{G}$ be the dual group of $G$: it has the same underlying set as $G$ and 
if $g \cdot h$ denotes the product in $G$ then the multiplication $\circ$ in $\hat{G}$
is defined by $g \circ h = h \cdot g$. The direct product $G \times \hat{G}$ will be important for the following
construction. Note that it is straightforward to define a black box for $G \times \hat{G}$ from a black box for $G$.

Recall from the introduction that  a derivation tree is acyclic if in every path from the root to a leaf every nonterminal appears at most once. 
The height of an acyclic derivation tree is bounded by $|N|$. 
We now define two important operations $\Delta$ and $\Gamma$. The operation $\Delta$ maps a tuple $s = (H_A)_{A \in N}$ of subgroups $H_A \leq G \times \hat{G}$ 
to a tuple $\Delta(s) = (L_A)_{A \in N}$ of subsets $L_A \subseteq G$ (not necessarily subgroups), whereas $\Gamma$ maps a tuple
$t = (L_A)_{A \in N}$ of subsets $L_A \subseteq G$ to a tuple $\Gamma(t) = (H_A)_{A \in N}$ of subgroups $H_A \leq G \times \hat{G}$.

We start with $\Delta$. Let $s=(H_A)_{A \in N}$ be a tuple of subgroups $H_A \leq G \times \hat{G}$. The tuple $\Delta(s) = 
(L_A)_{A \in N}$ of subsets $L_A \subseteq G$ is obtained as follows: Let $T$ be an acyclic derivation tree with root $r$ labelled by $A \in N$. 
We assign inductively a set $L_v \subseteq G$ to every inner node $v$: Let $B$ the label of $v$.
If $v$ has only one child it must be a leaf since our grammar is in Chomsky normal form. 
Let $g \in G$ be the label of this leaf. Then we set $L_v = \{h_1 g h_2 \mid (h_1,h_2) \in H_B \}$. If $v$ has two children $v_1,v_2$ (where $v_1$ is the left child and $v_2$ the right child),
then the sets $L_{v_1} \subseteq G$ and $L_{v_2} \subseteq G$ are already determined and we set 
$$
L_v = \{h_1g_1g_2h_2 \mid (h_1,h_2) \in H_B, g_1 \in L_{v_1}, g_2 \in L_{v_2}\}.
$$ 
We set $L(T) = L_r$ and finally define $L_A$ as the union of all sets $L(T)$ where $T$ is an acyclic derivation whose root is labelled with $A$.

The second operation $\Gamma$ is defined as follows: Let $t = (L_A)_{A \in N}$ be a tuple of subsets $L_A \subseteq G$. 
Then we define the tuple $\Gamma(t) = (H_A)_{A \in N}$ with $H_A \leq G \times \hat{G}$ as follows: Fix a nonterminal $A \in N$.
Consider a sequence $p = (A_i \to A_{i,0} A_{i,1})_{1 \le i \le m}$ of productions $(A_i \to A_{i,0} A_{i,1}) \in P$ and a sequence
$d = (d_i)_{1 \le i \leq m}$ of directions $d_i \in \{0,1\}$ such that 
$A_{i+1} = A_{i,d_i}$ for all $1 \leq i \leq m$, $A_1 = A = A_{m,d_m}$.
Basically, $p$ and $d$ define a path from $A$ back to $A$.
For every $1 \le i \leq m$ we define the sets
$$
M_i = \begin{cases} 
     L_{A_{i,0}} \times \{1\} \text{ if } d_i = 1 \\
     \{1\} \times L_{A_{i,1}}  \text{ if } d_i = 0
     \end{cases}
$$
We view $M_i$ as a subset of $G \times \hat{G}$ and define
$$
M(p,d) = \prod_{1 \le i \le m} M_i,
$$
where $\prod$ refers to the product in  $G \times \hat{G}$. If $p$ and $d$ are the empty sequences ($m=0$) then $M(p,d) = \{ (1,1) \}$.
Finally we define $H_A$ as the set of all $M(p,d)$, where $p = (A_i \to A_{i,0} A_{i,1})_{1 \le i \le m}$ and $d = (d_i)_{1 \le i \leq m}$ are as above (including
the empty sequences).
This set $H_A$ is a subgroup of $G \times \hat{G}$. To see this, it suffices to argue that $H_A$ is a monoid.
The latter follows from the fact that two pairs of sequences $(p,d)$ and $(p',d')$ of the above form can be composed to a single pair $(pp', dd')$.

In the following, we will speak of ${\bf NP}$ algorithms with oracles.
Here, we mean non-deterministic polynomial-time Turing machines with oracles.
However, the oracle can only be queried positively: There is an instruction
that succeeds if the oracle answers ``yes'', but cannot be executed if the
oracle would answer ``no''. This implies that if there is an ${\bf NP}$
(resp.\ ${\bf PSPACE}$ algorithm) for the oracle queries, there exists
an ${\bf NP}$ (resp.\ ${\bf PSPACE}$) algorithm for the entire
problem. Likewise, we will use the notion of oracle ${\bf PSPACE}$
algorithms.
\begin{lemma}\label{lemmacfg1}
	        For tuples $t=(L_A)_{A\in N}$, there is an ${\bf NP}$ algorithm for
        membership to the entries of $\Gamma(t)$, with access to an oracle for
        the entries of $t$.
\end{lemma}

\begin{proof}
Let $\Gamma(t) = (H_A)_{A \in N}$.
For every nonterminal $A \in N$ we define the NFA \[ \mathcal{A}_A = (N, (G \times \hat{G}), \{A\}, \delta, \{A\}),\] whose input alphabet is the finite group $G \times \hat{G}$.
The NFAs $\mathcal{A}_A$ only differ in the initial and final state.
The transition relation $\delta$ contains all triples $(B, (g,h), C) \in N \times (G \times \hat{G}) \times N$ such that for some $D \in N$
either $(B \to CD) \in P$, $g=1$, and $h \in L_D$ or $(B \to DC) \in P$, $h=1$, and $g \in L_D$.  Then we have $L(\mathcal{A}_A) = H_A$.
As in the spanning tree approach we add for every transition $(B,(g,h),C)$ in the NFA $\mathcal{A}_A$ also the inverse transition $(C,(g^{-1},h^{-1}),B)$.
In the following, $\mathcal{A}_A$ refers to this NFA. The number of transitions of the NFA $\mathcal{A}_A$ can be exponential in the input size, so we
cannot afford to construct $\mathcal{A}_A$ explicitly. But this is not necessary, since we only aim to come up with a nondeterministic polynomial time
algorithm.

Recall the spanning tree technique, which yields a generating set for the subgroup $L(\mathcal{A}_A) = H_A$.
This generating set will be in general of exponential size.
On the other hand, Lemma~\ref{lemma-lagrange} guarantees that the generating set produced by the spanning tree approach contains a subset of size at most
$\log_2 |G \times \hat{G}| = 2 \cdot \log_2 |G|$ that still generates $L(\mathcal{A}_A)$.
Note that $2 \cdot \log_2 |G|$ is linearly bounded in the input size. We can therefore nondeterministically produce a set of at most $2 \cdot \log_2 |G|$ loops in
the NFA $\mathcal{A}_A$. We do not even have to produce a spanning tree before: every generator produced by the spanning tree approach is a loop in
$\mathcal{A}_A$ and every such loop certainly yields an element of $H_A$. For every transition that appears on one of the guessed loops we guess
a transition label (either a pair $(1,h)$ or a pair $(g,1)$) and verify, using the oracle for membership to $L_B$,
that we guessed a transition in the NFA $\mathcal{A}_A$.

Let us denote with $S_A \subseteq G \times \hat{G}$ the set produced by the above nondeterministic algorithm.
For all nondeterministic choices we have $S_A \subseteq H_A$ and there exist nondeterministic choices for which
$\langle S_A \rangle = H_A$. By Lemma~\ref{lemma-gwp-black-box} we can finally check in {\bf NP} whether
a given pair $(g,h)$ belongs to $\langle S_A \rangle$.
\end{proof}

\begin{lemma}\label{lemmacfg2-NP}
Assume that the input grammar $\mathcal{G}$ is restricted to the class
$\text{CFG}(k)$ for some fixed $k$.  For tuples $s=(H_A)_{A\in N}$ of subgroups
$H_A\leq G\times\hat{G}$, there exists an ${\bf NP}$ algorithm for membership to
entries in $\Delta(s)$, with access to an oracle for membership to each entry
of $s$.
\end{lemma}

\begin{proof}
By assumption, the Horton-Strahler number of every acyclic derivation tree of $\mathcal{G}$ is bounded by the constant $k$. Since the height of an acyclic derivation tree is bounded by $|N|$ the total number of nodes in the tree is bounded by $2 |N|^k$ by Lemma~\ref{lemma-strahler}. Let $\Delta(s) = (L_A)_{A \in N}$. Fix an $A \in N$ and a group element $g \in G$.
We want to verify whether $g \in L_A$. For this we guess an acyclic derivation tree $T$ with root $A$. This can be done by a nondeterministic polynomial time machine.
Moreover we guess for every inner node $v$ of $T$ that is labelled with the nonterminal $B$ a pair $(h_{v,1}, h_{v,2}) \in G \times \hat{G}$ and verify using the oracle for membership to $H_A$ that
$(h_{v,1}, h_{v,2}) \in H_A$. If the verification is successful,
we evaluate every inner node $v$ to a group element $g_v \in G$. If $v$ has a single child, it must be labelled with a group element $h \in G$ (due to a production $B \to h$)
and we set $g_v = h_{v,1} h h_{v,2}$. If $v$ has two children $v_1$ (the left child) and $v_2$ (the right child) then we set $g_v = h_{v,1} g_{v_1} g_{v_2} h_{v,2}$.
At the end, we check whether $g = g_r$, where $r$ is the root of the tree $T$.
\end{proof}

\begin{lemma}\label{lemmacfg2-PSPACE}
For tuples $s = (H_A)_{A \in N}$ of subgroups $H_A \leq G \times \hat{G}$, there is
a ${\bf PSPACE}$ algorithm for membership to $\Delta(s)$, using an oracle for membership to the entries of $s$.
\end{lemma}

\begin{proof}
	The proof is similar to Lemma~\ref{lemmacfg2-NP}. However, without the restriction that the input grammar belongs to $\text{CFG}(k)$ for a fixed constant $k$,
an acyclic derivation tree of the grammar $\mathcal{G}$ may be of size exponential in the input length. But we will see that we never have to store the whole tree but only a polynomial sized part of the tree. To check $g \in L_A$ we do the following: We guess a production for $A$
and a pair $(h_1,h_2) \in G \times \hat{G}$ and verify using our oracle that $(h_1,h_2) \in H_A$.
If the guessed production for $A$ is of the form $A \rightarrow h$ for a group element $h \in G$ then we only have to check $h_1 h h_2 = g$ and we are done.
If the production is of the form $A \rightarrow BC$ for nonterminals $B,C \in N$ then we guess additional group elements $g_1,g_2 \in G$ and check that $g = h_1g_1g_2h_2$. If this holds, we continue with two recursive calls for  $g_1 \in L_B$ and $g_2 \in L_C$. We have to make sure that this eventually terminates. In order to ensure termination for every computation path we store the nonterminals that we already have seen. By this the recursion depth is bounded by $|N|$. This also ensures that we traverse an acyclic derivation tree.
We obtain a  nondeterministic polynomial space algorithm since the recursion depth is bounded by $|N|$ and the space used for the first recursive call can be reused for the second one.
\end{proof}

\begin{lemma}\label{lemmacfg3}
For tuples $s = (H_A)_{A \in N}$ subgroups $H_A \leq G \times \hat{G}$, there
exists an ${\bf NP}$ algorithm, with access to an oracle for membership to each
entry of $s$, with the following properties:
\begin{itemize}
\item On every computation path the machine outputs a tuple $(S_A)_{A \in N}$ of subsets $S_A \subseteq H_A$.
\item There is at least one computation path on which the machine outputs a tuple $(S_A)_{A \in N}$ such that every $S_A$ generates $H_A$.
\end{itemize}
\end{lemma}

\begin{proof}
By Lemma~\ref{lemma-lagrange} we know that every subgroup $H_A \leq G \times \hat{G}$ is generated by a set of at most $\log_2|G \times \hat{G}| = 2 \cdot \log_2|G|$ generators.
The machine simply guesses for every $A \in N$ a subset $R_A \subseteq G \times \hat{G}$ of size at most $2 \cdot \log_2|G|$. Then it verifies, using the oracle, for every $A \in N$ and every $(g,h) \in R_A$  $(g,h) \in H_A$. If all these verification steps succeed, the machine outputs the set $R_A$ for every $A \in N$.
\end{proof}
If membership for $H_A$ is in {\bf PSPACE} for every $A \in N$, then we could actually compute deterministically in polynomial space a generating set for every
$H_A$ by iterating over all elements of $G \times \hat{G}$. But we will not need this stronger fact.

\begin{proof}[Proofs of Theorem \ref{theoremPSPACEblack-box} and \ref{theoremStrahler-black-box}.]
Recall that $\phi_G : G^* \to G$ is the canonical morphism from Section~\ref{sec-prel}.
For every nonterminal $A \in N$ we define the subgroup $G_A \leq G \times \hat{G}$ by
$$
G_A = \{(\phi_G(u), \phi_G(v)) \mid u,v\in G^*, A  \Rightarrow^*_{\mathcal{G}} uAv\}.
$$
Recall that $L(A)$ is the set of group elements that can be produced from the nonterminal $A$.

\medskip
\noindent
{\em Claim 1.} $\Delta((G_A)_{A \in N}) = (L(A))_{A \in N}$.

\medskip
\noindent
To see this, let $\Delta((G_A)_{A \in N}) = (L_A)_{A \in N}$. The inclusion $L_A \subseteq L(A)$ is clear: the definition of $\Delta$ and $G_A$ 
directly yields a derivation tree with root labelled by $A$ for every element in $L_A$. For the inclusion $L(A) \subseteq L_A$ take an arbitrary
derivation tree $T$ for an element $w \in L(A)$ with root labelled by $A$. We can get an acyclic derivation tree from $T$ by contracting paths
from a $B$-labelled node down to another $B$-labelled node in $T$ for an arbitrary $B \in N$. If we choose these paths maximal, then they will
not overlap, which means that we can contract all chosen paths in parallel and thereby obtain an acyclic derivation tree.
 Each path produces a pair from $G_B$ for some $B \in N$. This shows that $w \in L_A$ and proves Claim 1.
 
\medskip
\noindent 
Let $s_0 =(H_A)_{A \in N}$ with $H_A = \{(1,1)\}$ for all $A \in N$
be the tuple of trivial subgroups of $G \times \hat{G}$.  For two tuples $s_1 = (H_{A,1})_{A \in N}$ 
and $s_2 = (H_{A,2})_{A \in N}$  of subgroups of $G \times \hat{G}$ we write $s_1 \leq s_2$ if $H_{A,1} \leq H_{A,2}$ for every $A \in N$.
By induction over $i \geq 0$ we show that $(\Gamma \Delta)^i(s_0) \leq (\Gamma \Delta)^{i+1}(s_0)$ for all $i$:
For $i=0$ this is clear and the induction step holds since $\Gamma$ as well as $\Delta$ are monotone with respect to componentwise
inclusion. Hence, we can define $\lim_{i \rightarrow \infty}(\Gamma \Delta)^i(s_0)$.

\medskip
\noindent
{\em Claim 2.}  $\displaystyle(G_A)_{A \in N} = \lim_{i \rightarrow \infty}(\Gamma \Delta)^i(s_0) = (\Gamma \Delta)^j(s_0)$ where $j = 2 \cdot |N| \cdot \lfloor \log_2 |G| \rfloor$.

\medskip
\noindent 
From the definition of $\Gamma$ and $\Delta$ we directly get 
$(\Gamma \Delta)^i(s_0) \le (G_A)_{A \in N}$ for every $i \ge 0$. Let us next show that
$(G_A)_{A \in N} \leq \lim_{i \rightarrow \infty}(\Gamma \Delta)^i(s_0)$. Let 
$\lim_{i \rightarrow \infty}(\Gamma \Delta)^i(s_0) = (H_A)_{A \in N}$ and $(g,h) \in G_A$. Hence, there exists a derivation $A  \Rightarrow^*_{\mathcal{G}} uAv$ such that $g = \phi_G(u)$
and $h = \phi_G(v)$. We prove $(g,h) \in H_A$ by induction on the length of this derivation.
Let $T$ be the derivation tree corresponding to the derivation $A  \Rightarrow^*_{\mathcal{G}} uAv$.
From the derivation $A  \Rightarrow^*_{\mathcal{G}} uAv$ we obtain a sequence
$p = (A_i \to A_{i,0} A_{i,1})_{1 \le i \le m}$ of productions $(A_i \to A_{i,0} A_{i,1}) \in P$ and a sequence
$d = (d_i)_{1 \le i \leq m}$ of directions $d_i \in \{0,1\}$ such that 
$A_{i+1} = A_{i,d_i}$ for all $1 \leq i \leq m$, $A_1 = A = A_{m,d_m}$.   
Assume that $A_{i,1-d_i}$ derives to $w_i \in G^*$ in the derivation $A  \Rightarrow^*_{\mathcal{G}} uAv$
 for all $1 \leq i \leq m$ and define 
 $$
 (u_i, v_i) = \begin{cases}
   (w_i, 1) & \text{ if } d_i = 1 \\
   (1, w_i) & \text{ if } d_i = 0 .
 \end{cases}
 $$
Then we obtain 
$$
(g,h) = \prod_{1 \le i \le m} (\phi_G(u_i), \phi_G(v_i))
$$
where the product is computed in the group $G \times \hat{G}$.
Let $T_i$ be the subtree of $T$ that corresponds to the derivation $A_{i,1-d_i} \Rightarrow^*_{\mathcal{G}} w_i$.
We now apply the same argument that we used for the proof of  Claim 1 to each of the trees $T_i$, i.e.,
we contract maximal subpaths from a $B$-labelled node down to a $B$-labelled node (for $B \in N$ arbitrary). 
Each of these subpaths corresponds to a derivation $B  \Rightarrow^*_{\mathcal{G}} u' B v'$ that is of course shorter
than the derivation $A  \Rightarrow^*_{\mathcal{G}} uAv$. By induction, we get $(\phi_G(u'), \phi_G(v')) \in G_B$.
Moreover, from the construction, it follows that (i) $\phi_G(w_i)$ belongs to the $A_{i,1-d_i}$-component of $\Delta((H_B)_{B \in N})$ and (ii) $(g,h)$
belongs to the $A$-component of $\Gamma(\Delta((H_B)_{B \in N}))$, which is $H_A$.
The construction is shown in Figure~\ref{fig-lemma-claim2}. All the paths between identical nonterminals in the subtrees below 
$A'_{1,1}, \ldots, A'_{6,0}$ are contracted and replaced by their``effects'', which by induction are already in the corresponding groups $G_X$ ($X = B, \ldots, G$).
This makes the subtrees acyclic.
This concludes the proof of the first equality in Claim 2.

\tikzset{snake it/.style={decorate, decoration={snake, amplitude =.2mm, segment length=1mm}}}
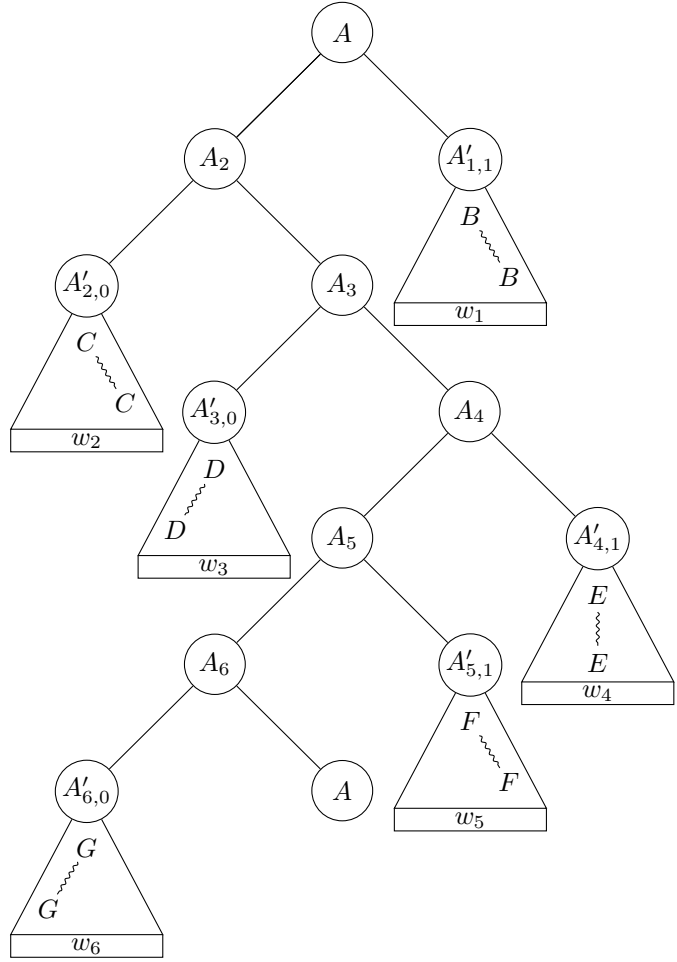
\begin{figure}[t]
\pgfkeys{/pgf/inner sep=.07em}  
  \centering
    \begin{tikzpicture}  
      \node[circle,draw,minimum width=8mm] (A1) {$A$} ;
      \node[circle,draw, below left = 1.1 cm and 1.1 cm of A1,minimum width=8mm] (A2) {$A_2$} ;
      \node[circle,draw, below right = 1.1 cm and 1.1 cm of A1,minimum width=8mm] (A'1) {$A'_{1,1}$} ;
      \node[circle,draw, below right = 1.1 cm and 1.1 cm of A2,minimum width=8mm] (A3) {$A_3$} ;
      \node[circle,draw, below left = 1.1 cm and 1.1 cm of A2,minimum width=8mm] (A'2) {$A'_{2,0}$} ;
      \node[circle,draw, below right = 1.1 cm and 1.1 cm of A3,minimum width=8mm] (A4) {$A_4$} ;
      \node[circle,draw, below left = 1.1 cm and 1.1 cm of A3,minimum width=8mm] (A'3) {$A'_{3,0}$} ;
      \node[circle,draw, below left = 1.1 cm and 1.1 cm of A4,minimum width=8mm] (A5) {$A_5$} ;
      \node[circle,draw, below right = 1.1 cm and 1.1 cm of A4,minimum width=8mm] (A'4) {$A'_{4,1}$} ;
      \node[circle,draw, below left = 1.1 cm and 1.1 cm of A5,minimum width=8mm] (A6) {$A_6$} ;
      \node[circle,draw, below right = 1.1 cm and 1.1 cm of A5,minimum width=8mm] (A'5) {$A'_{5,1}$} ;
      \node[circle,draw, below right = 1.1 cm and 1.1 cm of A6,minimum width=8mm] (A7) {$A$} ;
      \node[circle,draw, below left = 1.1 cm and 1.1 cm of A6,minimum width=8mm] (A'6) {$A'_{6,0}$} ;
            
     \draw (A1) -- (A2);
     \draw (A1) -- (A2);
     \draw (A1) -- (A'1);
     \draw (A2) -- (A3);
     \draw (A2) -- (A'2);
     \draw (A3) -- (A4);
     \draw (A3) -- (A'3);
      \draw (A4) -- (A5);
     \draw (A4) -- (A'4);
     \draw (A5) -- (A6);
     \draw (A5) -- (A'5);
     \draw (A6) -- (A7);
     \draw (A6) -- (A'6);
     
     \draw (A'1) -- ($(A'1)+(-1,-1.9)$) -- ($(A'1)+(1,-1.9)$) -- (A'1) ; 
     \draw ($(A'1)+(-1,-1.9)$) -- ($(A'1)+(-1,-2.2)$) -- node[pos=.5, label={[yshift=-.2mm]$w_1$}] {} ($(A'1)+(1,-2.2)$) -- ($(A'1)+(1,-1.9)$) ; 
     \draw (A'2) -- ($(A'2)+(-1,-1.9)$) -- ($(A'2)+(1,-1.9)$) -- (A'2) ; 
     \draw ($(A'2)+(-1,-1.9)$) -- ($(A'2)+(-1,-2.2)$) -- node[pos=.5, label={[yshift=-.2mm]$w_2$}] {} ($(A'2)+(1,-2.2)$) -- ($(A'2)+(1,-1.9)$) ; 
     \draw (A'3) -- ($(A'3)+(-1,-1.9)$) -- ($(A'3)+(1,-1.9)$) -- (A'3) ; 
     \draw ($(A'3)+(-1,-1.9)$) -- ($(A'3)+(-1,-2.2)$) -- node[pos=.5, label={[yshift=-.2mm]$w_3$}] {} ($(A'3)+(1,-2.2)$) -- ($(A'3)+(1,-1.9)$) ; 
     \draw (A'4) -- ($(A'4)+(-1,-1.9)$) -- ($(A'4)+(1,-1.9)$) -- (A'4) ; 
     \draw ($(A'4)+(-1,-1.9)$) -- ($(A'4)+(-1,-2.2)$) -- node[pos=.5, label={[yshift=-.2mm]$w_4$}] {} ($(A'4)+(1,-2.2)$) -- ($(A'4)+(1,-1.9)$) ; 
     \draw (A'5) -- ($(A'5)+(-1,-1.9)$) -- ($(A'5)+(1,-1.9)$) -- (A'5) ; 
     \draw ($(A'5)+(-1,-1.9)$) -- ($(A'5)+(-1,-2.2)$) -- node[pos=.5, label={[yshift=-.2mm]$w_5$}] {} ($(A'5)+(1,-2.2)$) -- ($(A'5)+(1,-1.9)$) ; 
     \draw (A'6) -- ($(A'6)+(-1,-1.9)$) -- ($(A'6)+(1,-1.9)$) -- (A'6) ; 
     \draw ($(A'6)+(-1,-1.9)$) -- ($(A'6)+(-1,-2.2)$) -- node[pos=.5, label={[yshift=-.2mm]$w_6$}] {} ($(A'6)+(1,-2.2)$) -- ($(A'6)+(1,-1.9)$) ; 
     
     \node[circle, below = 1 mm of A'1,minimum width=0mm] (B) {$B$} ;
     \node[circle, below right = 1.1 cm and .5mm of A'1,minimum width=1mm] (B') {$B$} ; 
     \path (B) edge[snake it] (B');
     
     \node[circle, below = 1 mm of A'2,minimum width=0mm] (C) {$C$} ;
     \node[circle, below right = 1.1 cm and .5mm of A'2,minimum width=1mm] (C') {$C$} ; 
     \path (C) edge[snake it] (C');
     
    \node[circle, below = 1 mm of A'3,minimum width=0mm] (D) {$D$} ;
     \node[circle, below left = 1.1 cm and .5mm of A'3,minimum width=1mm] (D') {$D$} ; 
     \path (D) edge[snake it] (D');
     
      \node[circle, below = 1 mm of A'4,minimum width=0mm] (E) {$E$} ;
     \node[circle, below = 1 cm of A'4,minimum width=1mm] (E') {$E$} ; 
     \path (E) edge[snake it] (E');
     
     \node[circle, below = 1 mm of A'5,minimum width=0mm] (F) {$F$} ;
     \node[circle, below right = 1.1 cm and .5mm of A'5,minimum width=1mm] (F') {$F$} ; 
     \path (F) edge[snake it] (F');
     
      \node[circle, below = 1 mm of A'6,minimum width=0mm] (G) {$G$} ;
     \node[circle, below left = 1.1 cm and .5mm of A'6,minimum width=1mm] (G') {$G$} ; 
     \path (G) edge[snake it] (G');
     \end{tikzpicture}
\caption{\label{fig-lemma-claim2} The situation in the proof of $(G_A)_{A \in N} \leq \lim_{i \rightarrow \infty}(\Gamma \Delta)^i(s_0)$.}
\end{figure}

Since all $G_A$ are finite groups there is a smallest number $j \geq 0$ such that
\begin{displaymath}
(\Gamma \Delta)^j(s_0) = (\Gamma \Delta)^{j+1}(s_0) .
\end{displaymath}
We then have $(\Gamma \Delta)^j(s_0) = \lim_{i \rightarrow \infty}(\Gamma \Delta)^i(s_0)$. It remains to show that $j \leq 2 \cdot |N| \cdot \log_2 |G|$.
In each component of the $|N|$-tuples $(\Gamma \Delta)^i(s_0)$ ($0 \leq i \leq j$)
we have a chain of subgroups of $G \times \hat{G}$. By Lagrange's theorem, any chain 
$\{ (1,1) \} = H_0 < H_1 < \cdots < H_{k-1} < H_k \leq G \times \hat{G}$ satisfies $k \leq 2 \cdot \log_2 |G|$. This shows that $j \leq 2 \cdot |N| \cdot \log_2 |G|$
and proves Claim 2.

\medskip
\noindent 
We can now prove Theorem~\ref{theoremStrahler-black-box}. 
By Claim 1 and Lemma~\ref{lemmacfg2-NP} it suffices to show that membership for the subgroups $G_A$ is in {\bf NP}.
For this, we construct a nondeterministic polynomial time machine that computes on every computation path a subset $S_A \subseteq G_A$ for every $A\in N$
such that on at least one computation path it computes a generating set for groups $G_A$ for all $A\in N$. Then we can decide membership for the $\langle S_A \rangle$ in {\bf NP}
by Lemma~\ref{lemma-gwp-black-box}.

The set $S_A$ is computed by initializing $S_A = \{ (1,1) \}$ for 
every $A \in N$ and then doing $2 \cdot |N| \cdot \log_2 |G|$ iterations of the following procedure: Assume that we have already produced the subsets $(S_A)_{A \in N}$.
Membership in $\left\langle S_A \right\rangle$ can be decided in {\bf NP} by Lemma~\ref{lemma-gwp-black-box}.
Hence, by Lemmas~\ref{lemmacfg1} and \ref{lemmacfg2-NP} one can decide membership in every entry of the tuple $\Gamma(\Delta((\left\langle S_A \right\rangle)_{A \in N}))$ in {\bf NP}.
Finally, by Lemma~\ref{lemmacfg3} we can produce nondeterministically in polynomial time a subset $S'_A \subseteq G \times \hat{G}$ for every $A \in N$ such that
for every computation path we have $(\left\langle S'_A \right\rangle)_{A \in N} \le \Gamma(\Delta((\left\langle S_A \right\rangle)_{A \in N}))$ and for at least one computation
path the machine produces subsets $S'_A$ with $(\left\langle S'_A \right\rangle)_{A \in N} = \Gamma(\Delta((\left\langle S_A \right\rangle)_{A \in N}))$.
With the sets $S'_A$ we go into the next iteration.
By Claim 2 there will be at least one computation path on which after 
$2 \cdot |N| \cdot \log_2 |G|$ iterations we get generating set for all the groups $G_A$.
This concludes the proof of Theorem~\ref{theoremStrahler-black-box}.

The proof of Theorem~\ref{theoremPSPACEblack-box} is identical except that we have to use Lemma~\ref{lemmacfg2-PSPACE} instead of 
Lemma~\ref{lemmacfg2-NP}.
\end{proof}

\section{Restrictions of  rational subset membership in symmetric groups}

In this section, we want to contrast the general upper bounds from the previous sections with lower bounds for symmetric groups and restricted versions
of the rational subset membership problem. We start with the subset sum problem.

\subsection{Subset sum in permutation groups}

The following result refers to the abelian group $\mathbb{Z}_3^m$, for which we use the additive notation.
\begin{theorem}
The following problem is {\bf NP}-hard:\\
Input: unary encoded number $m$ and a list of group elements $g,g_1, \ldots, g_n \in \mathbb{Z}_3^m$.\\
Question: Are there $i_1,\dots,i_n \in \{0,1\}$ such that $g= \sum_{1 \le k \le n} i_k \cdot g_k$?
\end{theorem}

\begin{proof}
We prove the theorem by a reduction from the problem {\em exact 3-hitting set problem} (X3HS):
\begin{problem}[X3HS]\label{1-in-3-SAT+problem}~\\
Input: a finite set $A$ and a set $\mathcal{B} \subseteq 2^A$ of subsets of $A$, all of size 3.\\
Question: Is there a subset $A' \subseteq A$ such that $|A' \cap C|=1$ for all $C \in \mathcal{B}$? 
\end{problem}
X3HS is the same problem as positive 1-in-3-SAT, which is {\bf NP}-complete \cite[Problem LO4]{gareyjohnson}. 

Let $A$ be a finite set and $\mathcal{B} \subseteq 2^A$ be a set of subsets of $A$, all of size $3$.
W.l.o.g. assume that $A = \{1,\ldots,n\}$ and let $\mathcal{B} = \{C_1, C_2, \ldots, C_m\}$.
We work in the group $\mathbb{Z}_3^m$. 
For every $1 \leq i \leq n$ let
$$
X_i = (a_{i,1}, a_{i,2}, \ldots, a_{i,m}) \in \mathbb{Z}_3^m,
$$
where 
$$
a_{i,j} = \begin{cases}
  0 \text{ if } i \notin C_j \\
  1 \text{ if } i \in C_j.
\end{cases}
$$
Then there exists $A' \subseteq A$ such that $|A' \cap C_j| = 1$ for every $1 \leq j \leq m$ 
if and only if the following equation has a solution $y_1, \ldots, y_n \in \{0,1\}$:
$$
\sum_{i=1}^n y_i \cdot X_i  = (1,1,\ldots,1).
$$
This proves the theorem.
\end{proof}
Clearly $\mathbb{Z}_3^m \leq S_{3m}$. We obtain the following corollary:

\begin{corollary} \label{coro-abelian-subsetsum}
The abelian subset sum problem for symmetric groups is {\bf NP}-hard.
\end{corollary}
Let us remark that the subset sum problem for $\mathbb{Z}_2^m$ (with $m$ part of the input) is equivalent
to the subgroup membership problem for $\mathbb{Z}_2^m$ (since every element of $\mathbb{Z}_2^m$ has order two)
and therefore can be solved in polynomial time.

\subsection{Knapsack in permutation groups}

We now come to the knapsack problem in permutation groups.
{\bf NP}-hardness of the general version of knapsack can be easily deduced from a result of Luks:

\begin{theorem}[\cite{luks}]\label{knapsackluks}
The knapsack problem for symmetric groups is {\bf NP}-hard.
\end{theorem}

\begin{proof}
Recall from the introduction that Luks \cite{luks} proved {\bf NP}-completeness of 3-membership for the special case of membership in a product $G H G$ where 
$G$ and $H$ are abelian subgroups of $S_m$.

Let us now assume that $G, H \leq S_m$ are abelian. Let $g_1,g_2,\dots,g_k$ be the given generators of $G$ and let $h_1,h_2,\dots,h_l$ be the given generators of $H$. 
Then $s \in GHG$ is equivalent to the solvability of the equation
\begin{displaymath}
s = g_1^{x_1}g_2^{x_2} \cdots g_k^{x_k}h_1^{y_{1}}h_2^{y_{2}} \cdots h_l^{y_{l}}g_1^{z_{1}}g_2^{z_{2}} \cdots g_k^{z_{k}}
\end{displaymath}
This is an instance of the knapsack problem, which is therefore {\bf NP}-hard.
\end{proof}
We next want to prove that already $3$-knapsack is {\bf NP}-hard.
In other words: the $k$-membership problem is {\bf NP}-hard for every $k \geq 3$ even if the groups are cyclic.
We prove this by a reduction from X3HS; see Problem~\ref{1-in-3-SAT+problem}.
For this, we need two lemmas.

Let $p>0$ be an integer.
For the rest of the section we write $[p]$ for the cycle $(1,2,\dots,p)$ mapping $p$ to $1$ and
$i$ to $i+1$ for $1 \leq i \leq p-1$.

\begin{lemma}\label{lemmaba}
Let $p,q \in \mathbb{N}$ such that $q$ is odd and $p > q > 0$ holds. 
Then the products $[p] [q]$ and $[q] [p]$ are cycles of length $p$.
\end{lemma}
\begin{proof}
Let $p$ and $q$ be as in the lemma. It is easy to verify that 
\begin{displaymath}
[q] [p] = (1,3,5,\dots,q-2,q,2,4,6,\dots,q-1,q+1,q+2,q+3,\dots,p),
\end{displaymath}
which is a cycle of length $p$. Because of
$[p] [q] = [q]^{-1}([q] [p]) [q]$, also $[p] [q]$ is a cycle of length $p$.
\end{proof}

\begin{lemma}\label{lemmaequation}
Let $p,q \in \mathbb{N}$ be primes such that $2 < q < p$ holds. Then 
\begin{equation}\label{equationlemmaequation}
[p]^{-x_2}[q]^{x_1}([p][q])^{x_2} = [q] = [q]^{x_1} [p]^{-x_2}([p][q])^{x_2}
\end{equation}
if and only if 
$(x_1 \equiv 1 \bmod q$ and $x_2 \equiv 0 \bmod p)$ or $(x_1 \equiv 0 \bmod q$ and $x_2 \equiv 1 \bmod p)$.
\end{lemma}

\begin{proof}
Let $p$ and $q$ be as in the lemma. By Lemma~\ref{lemmaba}, $[p][q]$ is a cycle of length $p$.
Therefore, $(x_1 \equiv 1 \bmod q$ and $x_2 \equiv 0 \bmod p)$ or $(x_1 \equiv 0 \bmod q$ and $x_2 \equiv 1 \bmod p)$ 
ensures that \eqref{equationlemmaequation} holds.

For the other direction, assume that $x_1$ and $x_2$ are such that \eqref{equationlemmaequation}  holds.
We obtain
\begin{equation} \label{eq-lemmaequation}
[p]^{-x_2}[q]^{x_1} = [q]^{x_1}[p]^{-x_2} .
\end{equation}
First of all we show that $x_1 \not \equiv 0 \bmod q$ implies $x_2 \equiv 0 \bmod p$. Assume that $x_1 \not \equiv 0 \bmod q$ and $x_2 \not\equiv 0 \bmod p$.
We will deduce a contradiction.
We first multiply both sides of \eqref{eq-lemmaequation} by $[p]^{x_2}$ and obtain
\begin{displaymath}
[q]^{x_1} = [p]^{x_2}[q]^{x_1} [p]^{-x_2} .
\end{displaymath}
Since $q$ is prime and $x_1 \not \equiv 0 \bmod q$ we can raise both sides to the power of $x_1^{-1} \bmod q$ and get
\begin{displaymath}
[q] = [p]^{x_2}[q] [p]^{-x_2}, 
\end{displaymath}
from which we obtain 
\begin{displaymath}
[q][p]^{x_2}[q]^{-1} = [p]^{x_2}
\end{displaymath}
by multiplying with $[p]^{x_2}[q]^{-1}$.
Since $x_2 \not \equiv 0 \bmod p$ and $p$ is prime, we can raise both sides to the power of $x_2^{-1} \bmod p$ which finally gives us
\begin{displaymath}
[q] [p] [q]^{-1} = [p] .
\end{displaymath}
By evaluating of both sides at position $p$ (recall that $p>q$) we get the contradiction
\begin{displaymath}
p^{[q] [p] [q]^{-1}} = p^{[p] [q]^{-1}} = 1^{[q]^{-1}} = q \neq 1 = p^{[p]} ,
\end{displaymath}
which shows that $x_1 \not \equiv 0 \bmod q$ implies $x_2 \equiv 0 \bmod p$.
Obviously $x_1 \equiv 0 \bmod q, x_2 \equiv 0 \bmod p$ is not a solution of  \eqref{equationlemmaequation}. This shows that $x_1 \not \equiv 0 \bmod q$ if and only if $x_2 \equiv 0 \bmod p$. It remains 
to exclude the cases $x_1 \equiv \gamma_1 \bmod q$ for $2 \leq \gamma_1 \leq q-1$ and $x_2 \equiv \gamma_2 \bmod p$ for $2 \leq \gamma_2 \leq p-1$. The equation
\begin{displaymath}
[p]^{-0} [q]^{x_1}([p] [q])^0 = [q] = [q]^{x_1} [p]^{-0}([p] [q])^0
\end{displaymath}
can only be true if $x_1 \equiv 1 \bmod q$. Hence it remains to show that the equation
\begin{displaymath}
[p]^{-x_2}[q]^0([p] [q])^{x_2} = [q] = [q]^0 [p]^{-x_2} ([p] [q])^{x_2}
\end{displaymath}
can only be true if $x_2 \equiv 1 \bmod p$. First we multiply with $([p] [q])^{-x_2}$ and get
\begin{displaymath}
[p]^{-x_2} = [q]([p][q])^{-x_2} .
\end{displaymath}
We obtain
\begin{equation*}
[p]^{-x_2} = [q][q]^{-1}[p]^{-1}([p][q])^{-x_2+1} = [p]^{-1}([p][q])^{-(x_2-1)} .
\end{equation*}
We multiply with $[p]$ and invert both sides:
\begin{displaymath}
[p]^{x_2-1} = ([p][q])^{x_2-1}
\end{displaymath}
Assume that this equation holds for some $x_2 \not \equiv 1 \bmod p$. By Lemma \ref{lemmaba} $[p][q]$ is a cycle of length $p$. Hence we can raise both sides to the power of $(x_2-1)^{-1} \bmod p$ and obtain $[p] = [p][q]$,  which is a contradiction since $[q] \neq 1$. This concludes the proof of the lemma.
\end{proof}

\begin{theorem}\label{knapsackn3}
The problem $3$-knapsack for symmetric groups is {\bf NP}-hard.
\end{theorem}

\begin{proof}
We provide a log-space reduction from the {\bf NP}-complete problem X3HS (Problem \ref{1-in-3-SAT+problem}) to 3-knapsack.  Let
$A$ be a finite set and $\mathcal{B} \subseteq 2^A$ such that every $C \in \mathcal{B}$ has size 3.
W.l.o.g.~let $A = \{1,\ldots, m\}$ and let $\mathcal{B} = \{C_1, C_2, \ldots, C_d\}$ where $C_i = \{ \alpha(i,1), \alpha(i,2), \alpha(i,3)\}$
for a mapping $\alpha:\{1,\dots,d\} \times \{1,2,3\} \rightarrow \{1,\dots,m\}$.

Let $p_1,\dots,p_m,r_1,\dots,r_m,q_1,\dots,q_d$ be the first $2m+d$ odd primes such that $p_j > r_j > 2$ and $p_j > q_i > 2$ for $1 \leq i \leq d$ and $1 \leq j \leq m$ hold. Moreover let $P = \max_{1\leq j \leq m} p_j$. Intuitively, the primes $p_j$ and $r_j$ ($1 \leq j \leq m$)  belong to $j \in A$ and the prime $q_i$ ($1 \leq i \leq d$) belongs  to the set $C_i$. 

We will work in the group
$$
G = \prod_{j=1}^m \mathcal{V}_j \times \prod_{i=1}^d \mathcal{C}_i ,
$$
where $\mathcal{V}_j \le S_{4p_j + r_j}$ and $\mathcal{C}_i \le S_{q_i+3P}$. 
More precisely we have
\begin{equation*}
\mathcal{V}_j = S_{p_j} \times S_{p_j} \times \mathbb{Z}_{p_j} \times \mathbb{Z}_{p_j} \times \mathbb{Z}_{r_j} \text{ and } 
\mathcal{C}_i = \mathbb{Z}_{q_i} \times S_{P} \times S_{P} \times S_{P}  .
\end{equation*}
In the following, we denote the identity element of a symmetric group $S_m$ with $\mathsf{id}$ in order to not confuse it with
the generator of a cyclic group $\mathbb{Z}_m$.

We now define four group elements $g,g_1,g_2,g_3 \in G$. We write 
$g = (v_1, \ldots, v_m, c_1, \ldots c_d)$ and $g_k = (v_{k,1}, \ldots, v_{k,m}, c_{k,1}, \ldots, c_{k,d})$
with $v_j,v_{k,j} \in \mathcal{V}_j$ and $c_i, c_{k,i} \in \mathcal{C}_i$. These elements are defined as follows:
\begin{alignat*}{9}
v_j      & \ = \  ([r_j],         & \; & [r_j],         &  \;    &0, &\;  &0, &\;  &0) \qquad  &  c_i      & \ = \ (1,&\; &\mathsf{id},                 &\; &\mathsf{id},                &\; &\mathsf{id}) \\
v_{1,j} & \ = \  ([r_j],         & \; & [p_j]^{-1}, & \;    &1, &\;  &1,  &\;  &1) \qquad  & c_{1,i} & \ = \ (1, &\; &[q_i]^{-1},                   &\; &[p_{\alpha(i,2)}]^{-1}, &\; &[q_i] [p_{\alpha(i,3)}]) \\
v_{2,j} & \ = \  ([p_j]^{-1}, & \; & [r_j],         & \;   -&1, &\;  &0,  &\; -&1) \qquad & c_{2,i} & \ = \ (1, &\; &[q_i] [p_{\alpha(i,1)}], &\; &[q_i]^{-1},                   &\; &[p_{\alpha(i,3)}]^{-1}) \\
v_{3,j} & \ = \  ([p_j][r_j],  & \; & [p_j][r_j],  & \,    &0, &\;  -&1, &\;  &0) \qquad  & c_{3,i} & \ = \ (1, &\; &[p_{\alpha(i,1)}]^{-1}, &\;  &[q_i] [p_{\alpha(i,2)}], &\; &[q_i]^{-1})   
\end{alignat*}
We claim that there is a subset $A' \subseteq A$ such that $|A' \cap C_i|=1$ for every $1 \le i \le d$ if and only if there are $z_1,z_2,z_3 \in \mathbb{Z}$ with 
\begin{equation*}
g = g_1^{z_1}g_2^{z_2}g_3^{z_3}
\end{equation*}
in the group $G$. Due to the direct product 
decomposition of $G$ and the above definition of $g, g_1, g_2, g_3$,
the statement $g = g_1^{z_1}g_2^{z_2}g_3^{z_3}$ is equivalent to the conjunctions of the following statements (read the above definitions of the 
$v_j,v_{k,j}, c_i, c_{k,i}$ columnwise) for all $1 \le j \le m$ and $1 \le i \le d$:
\begin{enumerate}[(a)]
\item $[r_j] = [r_j]^{z_1} [p_j]^{-z_2} ([p_j][r_j])^{z_3}$
\item $[r_j] = [p_j]^{-z_1} [r_j]^{z_2} ([p_j][r_j])^{z_3}$
\item $z_1 \equiv z_2 \bmod p_j$
\item $z_1 \equiv z_3 \bmod p_j$
\item $z_1 \equiv z_2 \bmod r_j$
\item $1 \equiv z_1+z_2+z_3 \bmod q_i$ 
\item $\mathsf{id} = [q_i]^{-z_1} ([q_i][p_{\alpha(i,1)}])^{z_2} [p_{\alpha(i,1)}]^{-z_3}$
\item $\mathsf{id} = [p_{\alpha(i,2)}]^{-z_1} [q_i]^{-z_2}([q_i][p_{\alpha(i,2)}])^{z_3}$
\item $\mathsf{id} = ([q_i][p_{\alpha(i,3)}])^{z_1} [p_{\alpha(i,3)}]^{-z_2} [q_i]^{-z_3}$
\end{enumerate}
Recall that by Lemma~\ref{lemmaba}, $[p_j][r_j]$ and $[q_i][p_j]$ are cycles of length $p_j$. Due to the congruences in (c), (d), and (e), the conjunction of (a)--(i) is equivalent
to the conjunction of the following equations:
\begin{enumerate}[(a)] \setcounter{enumi}{9}
\item $z_1 \equiv z_2 \equiv z_3 \bmod p_j$
\item $z_1 \equiv z_2 \bmod r_j$
\item $[p_j]^{-z_1} [r_j]^{z_2} ([p_j][r_j])^{z_1} = [r_j] = [r_j]^{z_2} [p_j]^{-z_1} ([p_j][r_j])^{z_1}$
\item $1 \equiv z_1+z_2+z_3 \bmod q_i$ 
\item $\mathsf{id} = [q_i]^{-z_1} ([q_i][p_{\alpha(i,1)}])^{z_1} [p_{\alpha(i,1)}]^{-z_1}$
\item $\mathsf{id} = [p_{\alpha(i,2)}]^{-z_1} [q_i]^{-z_2}([q_i][p_{\alpha(i,2)}])^{z_1}$
\item $\mathsf{id} = ([q_i][p_{\alpha(i,3)}])^{z_1} [p_{\alpha(i,3)}]^{-z_1} [q_i]^{-z_3}$
\end{enumerate}
By Lemma~\ref{lemmaequation}, the conjunction of (j)--(p) is equivalent to the conjunction of the following statements:
\begin{enumerate}[(a)] \setcounter{enumi}{16}
\item ($z_1 \equiv z_2 \equiv z_3 \equiv 0 \bmod p_j$ and $z_1 \equiv  z_2 \equiv 1 \bmod r_j$) or \\ ($z_1  \equiv z_2 \equiv z_3 \equiv 1 \bmod p_j$ and $z_1 \equiv  z_2 \equiv 0 \bmod r_j$)
\item $1 \equiv z_1+z_2+z_3 \bmod q_i$ 
\item $\mathsf{id} = [q_i]^{-z_1} ([q_i][p_{\alpha(i,1)}])^{z_1} [p_{\alpha(i,1)}]^{-z_1}$
\item $\mathsf{id} = [p_{\alpha(i,2)}]^{-z_1} [q_i]^{-z_2}([q_i][p_{\alpha(i,2)}])^{z_1}$
\item $\mathsf{id} = ([q_i][p_{\alpha(i,3)}])^{z_1} [p_{\alpha(i,3)}]^{-z_1} [q_i]^{-z_3}$
\end{enumerate}
Let us now assume that $A' \subseteq A$ is such that $|A' \cap C_i|=1$ for every $1 \le i \le d$.
Let $\sigma : \{1,\ldots,m\} \to \{0,1\}$ such that $\sigma(j) = 1$ iff $j \in A'$. Thus, $\alpha(i,1)+\alpha(i,2)+\alpha(i,3)=1$ for all $1 \leq i \leq d$.
By the Chinese remainder theorem, we can set $z_1, z_2, z_3 \in \mathbb{Z}$ such that 
\begin{itemize}
\item $z_1 \equiv z_2 \equiv z_3 \equiv \sigma(j) \bmod p_j$ and $z_1 \equiv  z_2 \equiv 1-\sigma(j) \bmod r_j$ for $1 \le j \le m$,
\item $z_k \equiv \sigma(\alpha(i,k)) \bmod q_i$ for $1 \le i \le d$ and $1 \le k \le 3$.
\end{itemize}
Then (q) and (r) hold. For (s), one has to consider two cases: if $\sigma(\alpha(i,1)) = 0$, then 
$z_1 \equiv 0 \bmod q_i$ and $z_1 \equiv 0 \bmod p_{\alpha(i,1)}$.
Hence, the right-hand side of (s) evaluates to 
$$[q_i]^{-0} ([q_i][p_{\alpha(i,1)}])^{0} [p_{\alpha(i,1)}]^{-0} = \mathsf{id}.
$$ 
On the other hand, if $\sigma(\alpha(i,1)) = 1$, then
$z_1 \equiv 1 \bmod q_i$ and $z_1 \equiv 1 \bmod p_{\alpha(i,1)}$ and 
the right-hand side of (s) evaluates again to 
$$
[q_i]^{-1} [q_i][p_{\alpha(i,1)}] [p_{\alpha(i,1)}]^{-1} = \mathsf{id}.
$$
In the same way, one can show that also (t) and (u) hold.

For the other direction, assume that $z_1, z_2, z_3 \in \mathbb{Z}$ are such that (q)--(u) hold.
We define $A' \subseteq \{1,\ldots,m\}$ such that for every $1 \leq j \leq m$:
\begin{itemize}
\item $j \notin A'$ if $z_1 \equiv z_2 \equiv z_3 \equiv 0 \bmod p_j$ and $z_1 \equiv  z_2 \equiv 1 \bmod r_j$, and
\item $j \in A'$ if $z_1 \equiv z_2 \equiv z_3 \equiv 1 \bmod p_j$ and $z_1 \equiv  z_2 \equiv 0 \bmod r_j$.
\end{itemize}
Consider a set $C_i = \{ \alpha(i,1), \alpha(i,2), \alpha(i,3)\}$.
From the equations (s), (t), and (u) we get for every $1 \le i \le d$ and $1 \le k \le 3$:
\begin{itemize}
\item if $z_1 \equiv 0 \bmod p_{\alpha(i,k)}$ then $z_k \equiv 0 \bmod q_i$
\item if $z_1 \equiv 1 \bmod p_{\alpha(i,k)}$ then $z_k \equiv 1 \bmod q_i$
\end{itemize}
Together with $1 \equiv z_1+z_2+z_3 \bmod q_i$  and $q_i \geq 3$, this implies that there must be exactly one $k \in \{1,2,3\}$ such that
$z_1 \equiv 1 \bmod p_{\alpha(i,k)}$. Hence, for every $1 \le i \le d$ there is exactly one $k \in \{1,2,3\}$ such that $\alpha(i,k) \in A'$, i.e.,
$|\{ \alpha(i,1), \alpha(i,2), \alpha(i,3)\} \cap A'| = 1$.
\end{proof}
Theorem~\ref{theoremnpcomplete} is an immediate consequence of Corollaries~\ref{thm-rat-black-box} and~\ref{coro-abelian-subsetsum} and
Theorem~\ref{knapsackn3}.

Theorem~\ref{knapsackn3} leads to the question what the exact complexity of the  $2$-knapsack problem for symmetric groups is.
Recall that the complexity of Luks' 2-membership problem is a famous open problem in the algorithmic theory of permutation groups.
The restriction of the 2-membership problem to cyclic groups is easier:

\begin{theorem} \label{thm-2-knapsack}
The $2$-knapsack problem for symmetric groups can be solved in polynomial time.
\end{theorem}
\begin{proof}
Let $a,a_1,a_2 \in S_m$ be permutations and let $A, A_1, A_2$ be the corresponding 
permutation matrices. Recall the definition of the Kronecker product of two $m$-dimensional
square matrices $A$ and $B$: $A \otimes B = (a_{i,j} \cdot B)_{1 \le i,j\leq m}$, so it is an $m^2$-dimensional square matrix
with the $m^2$ blocks $a_{i,j} \cdot B$ for $1 \le i,j\leq m$.
By \cite[Theorem~4]{bell}, the equation $a_1^{x_1}a_2^{x_2} = a$ is equivalent to
\begin{equation} \label{kronecker}
(A_2^T \otimes I_m)^{x_2}(I_m \otimes A_1)^{x_1}\text{vec}(I_m) = \text{vec}(A),
\end{equation}
where $\text{vec}(A) = (A_{1,1}, \ldots, A_{n,1}, A_{1,2}, \ldots, A_{n,2}, \ldots, A_{1,n}, \ldots, A_{n,n})^T$
is the $m^2$-di\-mensional column vector obtained from the matrix $A$ by stacking all columns of $A$
on top of each other and $I_m$ is the $m$-dimensional identity matrix.
The matrices $A_2^T \otimes I_m$ and $I_m \otimes A_1$ commute; see \cite{bell}.
By \cite[Theorem~1.4]{babai2} one can finally check in polynomial time whether
\eqref{kronecker} has a solution.
\end{proof}

\section{Application to intersection problems}

In this section we prove Theorems~\ref{theorem-intersection-CFL-DFA(k)} and \ref{theorem-intersection-CFL-DFA}.
The proofs of the two results are almost identical. Let us show how to deduce Theorem~\ref{theorem-intersection-CFL-DFA(k)} from
Theorem~\ref{theoremStrahler}. Let $\mathcal{G}$ be a grammar from $\text{CFG}(k)$ and let $\mathcal{A}_1, \ldots, \mathcal{A}_n$
be a list of group DFAs. Let $\mathcal{A}_i = (Q_i, \Sigma, q_{i,0}, \delta_i, F_i)$.
W.l.o.g.~we assume that the $Q_i$ are pairwise disjoint and let $Q = \bigcup_{1 \leq i \leq n} Q_i$. 
To every $a \in \Sigma$ we can associate a permutation $\pi_a \in S_Q$ by setting $\pi_a(q) = \delta_i(q,a)$ if $q \in Q_i$.
Let $\mathcal{G}' \in \text{CFG}(k)$ be the context-free grammar over the terminal alphabet $S_Q$ obtained by replacing in $\mathcal{G}$
every occurence of $a \in \Sigma$ by $\pi_a$.
Then, we have $L(\mathcal{G}) \cap \bigcap_{1 \le i \le n} L(\mathcal{A}_i) \neq \emptyset$ if and only
if there exists a permutation $\pi \in L(\mathcal{G}')$ such that $\pi(q_{i,0}) \in F_i$ for every $1 \leq i \leq n$.
We can nondeterministically guess such a permutation and check $\pi \in L(\mathcal{G}')$ in {\bf NP} using 
Theorem~\ref{theoremStrahler}. This proves the upper bound
from Theorems~\ref{theorem-intersection-CFL-DFA(k)}. The lower bound already holds for the case that $L(\mathcal{G}) = \Sigma^*$
\cite{BlondinKM16}.

The proof of the upper bound in Theorem~\ref{theorem-intersection-CFL-DFA} is identical to the above proof, 
except that we use Theorem~\ref{theoremPSPACEcomplete}. For the lower bound, notice that the {\bf PSPACE}-complete context-free membership problem 
for symmetric groups can be directly reduced to the intersection non-emptiness problem from Theorem~\ref{theorem-intersection-CFL-DFA} 
(several group DFAs and a single context-free grammar): Take a context-free gammar $\mathcal{G}$ over the terminal alphabet $S_m$.
Let $\{ \pi_1, \ldots, \pi_n \}$ be the permutations that appear as terminal symbols in $\mathcal{G}$.
Let $\mathcal{G}'$ be the context-free gammar obtained from $\mathcal{G}$ by replacing every occurrence of $\pi_i$ by a new terminal symbol $a_i$.
We construct $m$ group DFAs $\mathcal{A}_1, \ldots, \mathcal{A}_m$ over the terminal alphabet $\{a_1, \ldots, a_n\}$ and state set $\{1,\ldots,m\}$.
The initial and (unique) final state of $\mathcal{A}_i$ is $i$ and the transition function of every $\mathcal{A}_i$ is the same function $\delta$ with
$\delta(q,a_i) = q^{\pi_i}$ for $1 \leq q \leq m$. Then we have $L(\mathcal{G})$ contains the identity permutation if and only if 
$L(\mathcal{G}') \cap \bigcap_{1 \le i \le m} L(\mathcal{A}_i)$ is non-empty.

\bibliographystyle{plainurl}
\bibliography{bib}

\appendix

\section*{Appendix}

\section{Testing membership in CFG(k)}

\begin{lemma}\label{lemmacflk}
Let $\mathcal{G}=(N,T,P,S)$ be a context-free grammar. We have $L(\mathcal{G}) \neq \emptyset$ if and only if $\mathcal{G}$ has 
an acyclic derivation tree.
\end{lemma}

\begin{proof}
Clearly, if there is an acyclic derivation tree, then there is a derivation tree and hence $L(\mathcal{G}) = \emptyset$.
For the reverse implication note that an arbitrary derivation tree can be made acyclic (as in the proof of the pumping lemma for context-free languages). 
\end{proof}
The notion of a {\em partial acyclic derivation tree} is defined as the notion of an acyclic derivation tree except that leafs may be labelled with terminals or nonterminals.
 
\begin{theorem}
For every fixed $k \geq 1$, the problem of checking whether a given context-free grammar belongs to CFG$(k)$ is in {\bf coNP}.
\end{theorem}

\begin{proof}
Let $\mathcal{G}=(N,\Sigma,P,S)$ be a context-free grammar in Chomsky normal form. We have $\mathcal{G} \in \text{CFG}(k)$ if and only if for every acyclic derivation tree the Horton-Strahler number is at most $k$. By this we obtain $\mathcal{G} \not \in \text{CFG}(k)$ if and only if there is an acyclic derivation tree with Horton-Strahler number greater than $k$.  By Lemma~\ref{lemma-strahler}
this holds if and only if one of the following conditions holds:
\begin{itemize}
\item There is an acyclic derivation tree with at most $2|N|^k$ nodes and Horton-Strahler number greater than $k$.
\item There is an acyclic derivation tree with more than $2|N|^k$ nodes. 
\end{itemize}
The second statement holds if and only if there is a partial acyclic derivation tree $T$ with $2|N|^k < |T| \leq 2|N|^k+2$ ($|T|$ denotes the number of nodes of $T$)
and for every leave $v$ in $T$ that is labelled with a nonterminal $A$ there is an acyclic derivation tree $T_v$ of arbitrary size whose root is labelled with $A$
and which contains no nonterminal that has already appeared on the path from the root of $T$ to node $v$. This holds, since in an acyclic derivation tree with more than
 $2|N|^k$ nodes we can remove subtrees such that the resulting partial acyclic derivation tree $T'$ satisfies $2|N|^k < |T'| \leq 2|N|^k+2$.

These conditions can be checked in {\bf NP} as follows:
First, we guess an acyclic derivation tree with at most $2|N|^k$ nodes and compute in polynomial time its Horton-Strahler number $s$. If $s > k$ then we accept.
If $s \leq k$, then we guess a partial acyclic derivation tree $T$ with $2|N|^k < |T| \leq 2|N|^k+2$.  
For every leaf $v$ of $T$ that is labelled with a nonterminal $A$ we define the subgrammar $G_v = (N_v, T, P_v, A)$: let $A_1,\dots,A_d$ ($A_1 = S$, $A_d = A$) 
be the nonterminals that appear on the path from the root of $T$ to the leaf $v$. Then we set $N_v = N \setminus \{A_1,\dots,A_{d-1}\}$. Moreover, $P_v$ is obtained from $P$ by removing every production that contains one of the nonterminals $A_1,\dots,A_{d-1}$. Finally the algorithm verifies deterministically  in polynomial time whether $G_v$ has an acyclic derivation tree $T_v$ of arbitrary size that is rooted in $A$. By Lemma~\ref{lemmacflk} this holds if and only if $L(G_v) \neq \emptyset$.
\end{proof}

\end{document}